\documentclass[12pt, reqno]{amsart}

\usepackage{amsmath,amsfonts, amsthm, amssymb, color, cite}
\usepackage{color}
\usepackage{fancyhdr}
\usepackage{graphicx}
\usepackage{epstopdf}
\usepackage{lastpage}

\newtheorem{thm}{Theorem}[section]
\newtheorem{lma}{Lemma}[section]
\newtheorem{pro}{Proposition}[section]

\theoremstyle{definition}

\theoremstyle{remark}

\numberwithin{equation}{section}
\allowdisplaybreaks

\def\f{\frac}

\def\hf1{^\f{1}{1-\xi^2}}

\def\be{\begin{equation}}
\def\en{\end{equation}}
\def\bs{\begin{split}}
\def\es{\end{split}}
\def\ba{\begin{align}}
\def\ea{\end{align}}

\author[Lin Chang]{Lin Chang}
\address{Department of Mathematics, Handan University, Handan, Hebei, P.R.China.}
\email{changlin23@163.com}

\author[Yuling Xu*]{Yuling Xu*}
\address{Department of Applied Mathematics,The Hong Kong Polytechnic University,Hong Kong,P.R.China.}
\email{lingda.xu@polyu.edu.hk}

\renewcommand{\fancyhead}{}
\title{   C\MakeLowercase{onvergence} r\MakeLowercase{ate} \MakeLowercase{toward} r\MakeLowercase{arefaction} w\MakeLowercase{ave} \MakeLowercase{under} p\MakeLowercase{eriodic} p\MakeLowercase{erturbation} \MakeLowercase{for} g\MakeLowercase{eneralized} K\MakeLowercase{orteweg}-\MakeLowercase{de} V\MakeLowercase{ries}-B\MakeLowercase{urgers} \MakeLowercase{equation}          }

\keywords{Rarefaction wave,  Cauchy problem, Korteweg-de Vries-Burgers equation, Periodic perturbations, Asymptotic behavior}
\date{\today}
\thanks{*Corresponding author}

\usepackage{subfig}
\usepackage{graphicx}
\begin{document}
\begin{abstract}
In this paper, a rarefaction wave under space-periodic perturbation for the generalized Korteweg-de Vries-Burgers equation is considered. It is shown that if the initial perturbation around the rarefaction wave is suitably small, then the solution of the  system tends to the rarefaction wave. The stability of solution under periodic perturbation is an interesting and important problem since the perturbation keeps oscillating at the far fields.

\
The key of proof is to construct a suitable ansatz carrying the same oscillation as the solution. Then we can find cancellations between solutions and ansatz such that
the perturbation belongs to some Sobolev space. The nonlinear stability can be obtained by  the energy method.
\end{abstract}
\maketitle
\section{Introduction and main theorems.}
We consider the asymptotic behavior of solution to the Cauchy problem for the generalized Korteweg-de Vries Burgers equation
\begin{align}\label{O1.1}
u_{t}+ f(u)_{x}+ {\mu} u_{xxx}{ {=}} {\gamma} u_{xx},    &\ x\in \mathbb{R},\ t>0,
\end{align}
where  the flux $f(u)\in C^1$  is  strictly convex, ${\mu}>0$ the dispersive coefficient,  $ {\gamma}>0$   the viscosity. We are interested in the global asymptotic stability of the rarefaction wave to (\ref{O1.1}). In this paper, we consider a Cauchy problem of \eqref{O1.1}  with the initial data    satisfying
\begin{equation}\label{O1.2}
u_{0}(x):{ {=}} u {(x,0)}\rightarrow \left\{\begin{array}{l}
\bar{u}_{l}+w_{0l}(x), \quad \quad  x\rightarrow   -\infty,\\
\bar{u}_{r}+w_{0r}{  {(x)}}, \quad \quad x \rightarrow \infty,
\end{array} \quad\right.
\end{equation}
where $\bar{u}_{l}$,   $\bar{u}_{r}$    are constants. Function $ w_{0 i}(x) \in L^{\infty}(\mathbb{R})$   is a periodic function with period $p_i>0$,($i{ {=}}r,l$) satisfying
\begin{equation}\label{O1.3}
\frac{1}{p_{i}} \int_{0}^{p_{i}} w_{0i}{{(x)}}\mathrm{d} x{ {=}} 0.
\end{equation}

As we all know, if  $ w_{0l}(x)=w_{0r}(x)=0$, the asymptotic stability of \eqref{O1.1} with the initial-data \eqref{O1.2} has been extensively studied, regardless of whether $\bar{u}_{l}>\bar{u}_{r}, $     or $\bar{u}_{l}<\bar{u}_{r} $.  We list some previous works below.

\

For $\bar{u}_{l}>\bar{u}_{r}$, the KdV-Burgers equation   admits the viscous shock wave solution as $\gamma \gg 1$, see \cite{BS1985}\cite{BR1994}\cite{GH1967} \cite{NR1998}, Moreover,   the exponential time decay rate was further obtained in Yin-Zhao-Zhou \cite{YZZ2009} provided that the initial value converges to the shock exponentially at the far field. For $\bar{u}_{l}<\bar{u}_{r}$,  Wang-Zhu \cite{WZ2002} shows that  the KdV-Burgers equation   tends to
 the centered rarefaction wave $U^{r}(x/t),$ which will be  defined by (\ref{O2.3}) in Section \ref{section2}. We refer to\cite{DZ2007},\cite{YR2009} for some more interesting works.

However, the situation where ${w}_{0l}\left( x\right)  \neq  0$ and ${w}_{0r}\left( x\right) \neq0$ becomes significantly more complex, as the standard energy method fails. Specifically, we apply the energy method to study the stability of the rarefaction wave ${U}^{r}$, where the ``energy" of the perturbation``$u - {U}^{r}$" should belong to some Sobolev spaces such as ${H}^{1}\left( \mathbb{R}\right)$. However, the above method cannot be directly applied in this paper because $u - {U}^{r}$ oscillates at the far field and thus does not belong to any ${L}^{p}$ space for $p \geq  1$. By introducing a suitable ansatz, Chang \cite{C2024} has recently shown that the asymptotic behavior of the solution as time goes to infinity still holds in this case. Therefore, in this paper, we discuss the case where ${\bar{u}}_{l} < {\bar{u}}_{r}$.

 Our goal is to establish the nonlinear stability of the rarefaction wave for the Cauchy problem \eqref{O1.1}-\eqref{O1.2} under periodic perturbations. Roughly speaking, the solution not only exists globally but also tends to a rarefaction wave  as time goes to infinity. The precise statements of the main results are given in Theorem \ref{theorem201}  in Section 2.
\
 
 Motivated by \cite{C2024} \cite{XYY2019}  , we introduce a suitable ansatz $U(x,t)$, which has the same oscillations as the solution $u(x,t)$ at the far field, so that $u-U^r$  belongs to some Sobolev spaces and the anti-derivative method is still available.

The ansatz is defined as ${U}= u_{l}{(x,t)}g(x,t) + u _ { r }  {(x,t)} [ 1 - g   ( x,t ) ]$, where $u_{ l }$ is a periodic solution of \eqref{O1.1} with the initial data $\bar{u}_{l}+w_{0 l}{{(x)}}$ in \eqref{O1.2} and is expected to have the same oscillation as $u$ near $x= -\infty$. Similarly $u_r$ is expected to have the same oscillation as $u$ near $x=\infty$. And $g(x,t)$ is a weighted function satisfying $\mathop{\lim }\limits_{{x \rightarrow   - \infty }}g\left( {x,t}\right)  = 1,\mathop{\lim }\limits_{{x \rightarrow   + \infty }}g\left( {x,t}\right)  = 0,$ which will be determined later.   Thus the perturbation $u - U(x,t)$ could belong to ${L}^{2}$.

The rest of the paper will be arranged as follows. In Section \ref{section2}, a suitable ansatz is constructed and the main results are stated. In Section \ref{section3},  the stability problem is reformulated to a perturbation equation around the ansatz. In Section \ref{section4},   The a priori estimate is established.   In Section \ref{section5},   the main result  is proved.
\noindent

\textbf { Notation.} The functional $\|\cdot\|_{L^p(\Omega)}$ is defined by $\| f\|_{L^p(\Omega)} =  (\int_{\Omega}|f|^{p}(x)\mathrm{d}x)^{\frac{1}{p}}$. The symbol $\Omega$ is often omitted  when $\Omega=(-\infty,\infty)$. We denote for simplicity
\begin{equation*}
\| f\| =  \left(\int_{ -\infty}^{ \infty}f^{2}(x)\mathrm{d}x\right)^{\frac{1}{2}}
\end{equation*}
as $p=2$. In addition, $H^m$ denotes the  $m$-th  order Sobolev space of functions defined by
 \begin{equation*}
\|f\|_{m} =  \left( \sum_{k=0}^{m}  \|\partial^{k}_{\xi}f\|^2 \right)^{\frac{1}{2}}.
\end{equation*}

\section{Preliminaries and  Main Results}\label{section2}
\subsection{ Suitable   Ansatz}

We consider the following Riemann problem:

\begin{align}\label{O2.1}
	w^{r}_{t}+   \frac{1}{2}\left[(w^{r})^{2}\right]_{x}=0,    &\ x\in \mathbb{R},\ t>0,
\end{align}

with the initial data

\begin{equation}\label{sO1.2}
	w{(x,0)}= w_0^r(x)= \left\{\begin{array}{l}
		\bar{w}_{l}, \quad \quad  x<0,\\
		\bar{w}_{r}, \quad \quad x >0.
	\end{array} \quad\right.
\end{equation}
When $\bar{w}_{l}< \bar{w}_{r}$, it is well-known that the solution of the Riemann problem (\ref{O2.1}),(\ref{sO1.2}) is the centered rarefaction wave ${w}^{r}\left( {x,t}\right)  = {w}^{r}\left( {x/t}\right)$, where
\begin{equation*}
	{w}^{r}\left( {x/t}\right)  = \left\{  \begin{array}{ll} \bar{w}_{l}, & x \leq  \bar{w}_{l}t, \\
		x/t, & \bar{w}_{l}t < x < \bar{w}_{r}t, \\
		\bar{w}_{r}, & x \geq  \bar{w}_{r}t.
	\end{array}\right.
\end{equation*}
Hence ${u}^{r}\left( {x,t}\right)$ is rewritten as ${u}^{r}\left( {x/t}\right)  = {\left( {f}^{\prime }\right) }^{-1}\left( {{w}^{r}\left( {x/t}\right) }\right)$ with $\bar{w}_{ l,r } = {f}^{\prime }\left( \bar{u}_{ l,r }\right)$. To analyze the stability of this wave, we construct a smooth approximation   solution of the Riemann solution ${w}^{r}\left( {x/t}\right)$ is constructed as follows:
\begin{equation*}
	\left\{  \begin{array}{l} {W}_{t} + W{W}_{x} = 0, \\  W\left( {x,0}\right)  = {W}_{0}\left( x\right)  = \frac{1}{2}\left( {\bar{w}_{r} + \bar{w}_{l}}\right)  + \widetilde{w}{K}_{q}{\int }_{0}^{\varepsilon x}{\left( 1 + {y}^{2}\right) }^{-q}\mathrm{d} y, \end{array}\right.  
\end{equation*}
where  $\widetilde{w} = \frac{1}{2}\left( {\bar{w}_{r} - \bar{w}_{l}}\right)  > 0,\varepsilon  > 0$ is a constant. And ${K}_{q}$ is a positive constant such that
\begin{equation*}
	{K}_{q}\int_{-\infty}^{\infty }\frac{\mathrm{d}y}{{\left( 1 + {y}^{2}\right) }^{q}} = 1\;\left( {q > \frac{1}{2}}\right) .
\end{equation*}
Introduce  ${U^{r}}\left( {x,t}\right)$ by ${U^{r}}\left( {x,t}\right)  = {\left( {f}^{\prime }\right) }^{-1}\left( {{w^r}\left( {x,t}\right) }\right)$, satisfying 
\begin{equation}
	\left\{  \begin{array}{l} U^{r}_{t} + f{\left( {U^{r}}\right) }_{x} = 0, \\  {\left. {U^{r}}\right| }_{t = 0} = U^{r}_{0}\left( x\right)  = {\left( {f}^{\prime }\right) }^{-1}\left( {W_{0}\left( x\right) }\right)  \rightarrow  \bar{u}_{ l,r  },x \rightarrow   \pm  \infty . \end{array}\right.  \label{O2.3}
\end{equation}
\label{yl2.1}\begin{lma}\cite[Lemma 2.2]{WZ2002}Setting $\varepsilon  =  \frac{1}{2}\left( {\bar{u}_{ r} - \bar{u}_{ l}}\right),q = 1$, the problem (\ref{O2.3}) has a unique global smooth solution ${U^{r}}\left( {x,t}\right)$ satisfying:
	\begin{itemize}
		\item [\rm(i)]   $\bar{u}_{l } < {U^{r}}\left( {x,t}\right)  < \bar{u}_{ r},U^{r}_{x}\left( {x,t}\right)  > 0$ for each $\left( {x,t}\right)  \in  \mathbb{R} \times  {\mathbb{R}}^{ + }.$
		\item [\rm(ii)]  For any $p$ with $1 \leq  p \leq  \infty$ , there exists a constant ${C}_{p}$ depending on $p$ such that
		\begin{equation*}
			\begin{split}
				&{\left\|U^{r}_{x}\right\|{}}_{{L}^{p}}^{p} \leq  h\left( \varepsilon\right) {C}_{p}{\left( 1 + t\right) }^{-p + 1},\\
				&{\left\|U^{r}_{xx}\right\|{}}_{{L}^{p}}^{p} \leq  h\left( \varepsilon\right) {C}_{p}{\left( 1 + t\right) }^{-\frac{{3p} - 1}{2}},\\
				&{\left\|U^{r}_{xxx}\right\|{}}_{{L}^{p}}^{p} \leq  h\left( \varepsilon\right) {C}_{p}{\left( 1 + t\right) }^{-{2p} + 1},\\
				&\| U^{r}_{xxxx}{\| }_{{L}^{p}}^{p} \leq  h\left( \varepsilon\right) {C}_{p}{\left( 1 + t\right) }^{-\frac{{5p} - 3}{2}},\\
				&{\left\| {(U^{r}_{xx})^{2}}(U^{r}_{x})^{-1}\right\|}_{{L}^{1}}   \leq  {Ch}\left( \varepsilon\right) {\left( 1 + t\right) }^{-\frac{3}{2}},\\
				&{\left\| {({U^{r}_{xxx}}^{2})}{(U^{r}_{x})^{-1}}\right\|}_{{L}^{1}}  \leq  {Ch}\left( \varepsilon\right) {\left( 1 + t\right) }^{-2},\\
			\end{split}
		\end{equation*}
		where $h\left( \varepsilon\right)$ is a function of $\varepsilon$ and satisfies $\mathop{\lim }\limits_{{\varepsilon \rightarrow  0}}h\left( \varepsilon\right)  = 0 .$ 
		\item [\rm(iii)] 
		\begin{equation*}  
			{\left\| {\partial }_{t}^{j}{\partial }_{x}^{k}{U^{r}}\right\| }_{L^{\infty} } \leq  C{\left| \bar{u}_{ r} - \bar{u}_{l }\right| }^{j + k + 1},j,k \geq  0,j + k \leq  4.
		\end{equation*}
		\item [\rm(iv)]  
		\begin{equation*}
			\mathop{\sup }\limits_{\mathbb{R}}\left| {{U^{r}}\left( {x,t}\right)  - {u}^{r}\left( {x/t}\right) }\right|  \rightarrow  0\text{ as} t \rightarrow  \infty.
		\end{equation*}
	\end{itemize}
\end{lma}
\subsection{Suitable Ansatz}
We construct weighted functions below. We assume that the  function  $u_i {(x,t)} $  is   the   solution  of \eqref{O2.1} with the  initial data ($ i{  {=}}l,r$):
\begin{equation*}
	u_{ i 0 } ^ { } (x):=\bar { u }_{ i} + w _{ 0 i } (x).
\end{equation*}
Function $g(x,t)$ is introduced by
\begin{equation}
	g(x,t)  := \frac {  {U^{r}} ( {x,t} ) - \bar { u }_{ r } } { \bar { u }_{ l } - \bar { u }_{ r } },\label{x2.5}
\end{equation}
satisfying
\begin{equation}
	\lim_{x\rightarrow -\infty} g(x,t)=1, \lim_{x\rightarrow +\infty} g  (x ,t )=0,\label{y2.5}
\end{equation}
\label{yl2}\begin{lma}\cite[Lemma 2.3]{C2024}
	Assume that $u_0\in H^{k+1} (0,p)  $  is a periodic function with period $p>0$ for any integer $k\geq0$. Then the periodic solution $u(x,t)$ of \eqref{O1.1} satisfies
	\begin{equation*}
		\begin{array}{l}
			\left\| \partial_{{ {x}}}^{k}  (u-\bar{u})  \right\|_{L^{\infty}(\mathbb{R})} \leq C  \|u_{0}-\bar{u} \|_{H^{k+1}(0,p)} e^{-\beta t}, \quad t \geqslant 0,
		\end{array}
	\end{equation*}
	where $\bar{u}=\frac{1}{p}\int_{0}^{p}u_0({{x}})d{{x}} $ and the positive constants $ C$, $\beta$  are independent of time $t$.
\end{lma}
Motivated by \cite{XYY2019}, we construct an ansatz below
\begin{equation}\label{x2.6}
	{U}{(x,t)}:= u_{l}{(x,t)}g{(x,t)}+u_{r}{(x,t)}[1-g{(x,t)}].
\end{equation}
Note that $u_{i}, i=l,r$ is a periodic solution of \eqref{O1.1} with the initial data $\bar{u}_{i}+w_{0 i}{{(x)}}$ and is expected to have the same oscillation as $u$ near $x=\mp\infty$, respectively.   Thus ${U}{(x,t)}$ is expected to have the same oscillation as $u{(x,t)}$ at the far fields. Now we begin to  study the property of the shift. Since $U$ is not the solution of the KdV-Burgers  equation \eqref{O1.1}, the error term is
\begin{equation*}\label{x2.8}
	h:={U}_{t} +f({U})_{{{x}}} .
\end{equation*}

\begin{lma} \label{yl4}Under the assumptions of Lemma \ref{yl2}, there  exists a small constant $\epsilon_{0}$, such that if    ${{\epsilon}}:=\max \|w_{0i} \|_{H^1(\Omega)}<{{\epsilon}}_{0};i=l,r.$, one gets that
	$$  \left\|\frac{\partial^{j}}{\partial x^{j}}(U-U^{r}),h\right\|_{{L}^{1}}, \left\|\frac{\partial^{j}}{\partial x^{j}}(U-U^{r}),h\right\|\leq  {C\epsilon }{e}^{-{\beta t}},\;t \geq  0,   j=0,1,2,3$$
	where $C > 0$ is independent of $\epsilon$ and $t$, and $\beta  > 0$ is the constant given in Lemma \ref{yl2}.
\end{lma}
The proof is postponed to the appendix due to its technical nature. We can define the perturbation by
\begin{equation*}
	\phi \left( {x,0}\right)  = u\left( {x,0}\right)  - U \left( {x,0}\right) ,
\end{equation*}
so that $ {{\phi }_{0} } $ belongs to some Sobolev space. We assume that the initial data satisfies
\begin{equation*}
	{{\phi }_{0} }  \left( x\right)  := {\phi  }  \left( {x,0}\right)  \in  {H}^{1}\left( \mathbb{R}\right) .
\end{equation*}
The main result is
\subsection{Main Theorem}
\begin{thm}\label{theorem201}   Suppose $ {{\bar{u}}_{l}} $ and $  {{\bar{u}}_{r}} $ can be connected by $ {{U}^{r}}(x,t)$, if there exists a positive constant $\eta  > 0$ such that ${\|{\phi}_{0}\|}_{1}^{2} + h\left( \varepsilon\right) +\epsilon \leq  \eta.$
	then the problem (\ref{O1.1})-(\ref{O1.3}) admits a unique global solution satisfying
	\begin{equation}\label{O2.7}
		\begin{split}
			&u - U \in  {C}^{0}(      \left[ 0, + \infty \right) ;{H}^{1}), \\
			&{\left( u - U\right) }_{x} \in  {L}^{2}(                    \left[ 0, + \infty \right) ;{H}^{1}),
		\end{split}
	\end{equation}
	and
	\begin{equation} \label{O2.8}
		\mathop{\sup }\limits_{{x \in  \mathbb{R}}}\left| {\left( {u - {U}^{r}}\right) \left( {x,t}\right) }\right|  \rightarrow  0,\forall t \rightarrow  \infty .
	\end{equation}
\end{thm}
\section{Reformulation of the Problem}\label{section3}
In this section, we reformulate our problem (\ref{O1.1}) in terms of the deviation from the asymptotic state. Now letting
\begin{equation*}
	u\left( {t,x}\right)  = {U} \left( {t,x}\right)  + \phi \left( {t,x}\right) , \label{3.1}
\end{equation*}
we reformulate the problem (\ref{O1.1}) in terms of the deviation $\phi$ from ${U} $ as
\begin{equation}
	\left\{  \begin{array}{l}  \phi_{t}  +  \left( {f\left( {\phi  + {U} }\right)  - f\left( {U} \right) }\right)_{x} - \mu \phi_{xx}  + \gamma \phi_{xxx}  = \mu  {U}_{xx}  - \gamma  {U}_{xxx}  -h\;\left( {t > 0,x \in  \mathbb{R}}\right) ,\\
		\phi \left( {0,x}\right)  = {\phi }_{0}\left( x\right)  :=  {u}_{0}\left( x\right)  - {U} \left( {0,x}\right)  \rightarrow  0\;\left( {x \rightarrow   \pm  \infty }\right) , \end{array}\right.  \label{O3.1}
\end{equation}
We will seek the solutions in the functional space $X\left( {t_1,t_2}\right)$ for any $0 \leq t_1<t_2 < \infty,$
\begin{equation*}\label{s3.3}
	{X}_{\delta }\left( {{t}_{1},{t}_{2}}\right)  = \left\{  {\phi \mid  \phi \in  C\left( {\left\lbrack  {{t}_{1},{t}_{2}}\right\rbrack  ;{H}^{1}}\right) ,{\phi}_{x} \in  {L}^{2}\left( {\left( {{t}_{1},{t}_{2}}\right) ;{H}^{1}}\right) \text{ and }\mathop{\sup }\limits_{\left\lbrack  {t}_{1},{t}_{2}\right\rbrack  }\| \phi\left( t\right) {\| }_{1} \leq  \delta }\right\}
\end{equation*}
where the constant $\delta  \ll1$ is small. In the above solution space, we have the following local existence theorem.
\begin{pro}\cite[Lemma 3.1]{C2024}(Local existence){\label{pp1}} Let ${\phi}_{0}\left( x\right)  \in  {H}^{1}\left( \mathbb{R}\right)$ and there exists $\delta > 0$ such that ${\|{\phi}_{0}\|}_{1} \leq\delta$. Then there exists a positive constant ${t}_{0}$ depending on $\delta$ such that the problem (\ref{O3.1}) admits a unique solution $\phi\left( {x,t}\right)$ in ${X}_{2\delta}\left( {0,{t}_{0}}\right)$.
\end{pro}
\begin{pro} {\label{pp3}} $(a$ priori estimate)Suppose that $\phi\left( {x,t}\right)$ is a solution of (\ref{O3.1}) in ${X}_{2\delta}\left( {0,T}\right)$ for positive constant $T$ and $\delta$. Then it holds that
	\begin{equation*}
		\| \phi\left( t\right) \|^{2}_{1} + \int_{0}^{t}\left( {{\|\sqrt{{U}^{r}_{x}}\phi\left( \tau \right) \|}^{2} + {\|{\phi}_{x}\left( \tau \right) \|}_{1}^{2}}\right) {\mathrm{d} \tau } \leq  {C}\left( {{\|{\phi}_{0}\|}_{1}^{2} + h\left( \varepsilon\right) +\epsilon}\right) ,
	\end{equation*}
\end{pro}
Combining proposition \ref{pp1}  and proposition \ref{pp3}, we get the following Theorem.
\begin{thm} \label{thm1} Suppose that there exists a positive constant $\eta  > 0$ such that ${\|{\phi}_{0}\|}_{1}^{2} + h\left( \varepsilon\right) +\epsilon \leq  \eta$. If ${\phi}_{0}\left( x\right)  \in  {H}^{1}$, then problem (\ref{O3.1}) has a unique global solution $\phi \in  C\left( {\lbrack 0,\infty }\right) ;{H}^{1})$  satisfying
	\begin{equation}\label{O3.2}
		\mathop{\sup }\limits_{{t\geq0}}\| \phi\left( t\right) \|^{2}_{1} + \int_{0}^{\infty}\left( {{\|\sqrt{{U}^{r}_{x}}\phi\left( \tau \right) \|}^{2} + {\|{\phi}_{x}\left( \tau \right) \|}_{1}^{2}}\right) {\mathrm{d} \tau } \leq  {C}\left( {{\|{\phi}_{0}\|}_{1}^{2} + h\left( \varepsilon\right) +\epsilon}\right).
	\end{equation}
\end{thm}
\section{ A priori estimate}\label{section4}
In this section, we show the following a priori estimate for $\phi$
\begin{lma}\label{ly7}  With the same assumption of Proposition \ref{pp3}, It follows  that
	\begin{equation*}
		\| \phi\left( t\right) {\| }^{2} + \int_{0}^{t}\left( {{\|\sqrt{{U}^{r}_{x}}\phi\left( \tau \right) \|}^{2} + {\|{\phi}_{x}\left( \tau \right) \|}^{2}}\right) {\mathrm{d} \tau } \leq  {C}_{0}\left( {{\|{\phi}_{0}\|}^{2} + h\left( \varepsilon\right) +\epsilon}\right) ,
	\end{equation*}
\end{lma}

\begin{proof} Multiplying $(\ref{O3.1})_1$ by $\phi$, one gets that
	\begin{equation*}
		\begin{split}
			&\frac{1}{2}\left(\phi^{2}\right)_{t}  +  \left( {f\left( {\phi  + {U^{r}} }\right)  - f\left( {U^{r}} \right) }\right)_{x} \phi+ \mu \phi_{x}^{2} \\
			=&\mu(\phi\phi_x)_x-  \gamma \left(\phi_{xx}\phi-\frac{1}{2}\phi_x^2\right)_{x}  + \mu  {U^{r}}_{xx}  \phi    - \gamma  {U^{r}}_{xxx}  \phi \\
			&+\left\{-h + \left\{\left( {f\left( {\phi  + {U^{r}} }\right)  - f\left( {U^{r}} \right) }\right)-\left( {f\left( {\phi  + {U} }\right)  - f\left( {U} \right) }\right)                                 \right\}_{x}\right\} \phi \\
			&+\left\{ \mu  (U-U^{r})_{xx}       - \gamma  (U-U^{r})_{xxx}  \right\}\phi \\
			:=&\mu(\phi\phi_x)_x-  \gamma \left(\phi_{xx}\phi-\frac{1}{2}\phi_x^2\right)_{x}  + \mu  {U^{r}}_{xx}  \phi    - \gamma  {U^{r}}_{xxx}  \phi +E \phi .\\
		\end{split}
	\end{equation*}
	Integrating the above equality with respect to $x$ over $\mathbb{R}$, we have

\begin{equation}\label{{O4.1}}
	\begin{split}
		&\frac{1}{2}\frac{\mathrm{d}}{\mathrm{d}t}\int_{-\infty}^{\infty}{\phi}^{2}\mathrm{d} x + \int_{-\infty}^{\infty}\phi{\left( f\left( \phi + U^{r}\right)  - f\left( U^{r}\right) \right) }_{x}\mathrm{d} x + \mu \int_{-\infty}^{\infty}{\phi}_{x}^{2}\mathrm{d} x \\
		=& \mu \int_{-\infty}^{\infty}\phi{U^{r}}_{xx}\mathrm{d} x + \gamma \int_{-\infty}^{\infty}\phi{U^{r}}_{xxx}\mathrm{d} x+\int_{-\infty}^{\infty} E\phi\mathrm{d} x.:=\sum^{3}_{i=1}A_i
	\end{split}
\end{equation}

The second term of (\ref{{O4.1}}) is estimated as follows:
\begin{equation}\label{{O4.2}}
	\begin{split}
		&\int_{-\infty}^{\infty}\phi{\left( f\left( \phi + {U^{r}}\right)  - f\left( {U^{r}}\right) \right) }_{x}\mathrm{d} x\\
		= &\int_{-\infty}^{\infty}\left\{  {-{\left( \int_{U^{r}}^{\phi + {U^{r}}}f\left( s\right) \mathrm{d}s - f\left( {U^{r}}\right) \phi\right) }_{x} + \left( {f\left( {\phi + {U^{r}}}\right)  - f\left( {U^{r}}\right)  - {f}^{\prime }\left( {U^{r}}\right) \phi}\right) U^{r}_{x}}\right\}  \mathrm{d}x\\
		=& \int_{-\infty}^{\infty}\left( {f\left( {\phi + {U^{r}}}\right)  - f\left( {U^{r}}\right)  - {f}^{\prime }\left( {U^{r}}\right) \phi}\right) U^{r}_{x}\mathrm{d} x\\
		\geq&  \frac{\alpha }{2}\int_{-\infty}^{\infty}U^{r}_{x}{\phi}^{2}\mathrm{d} x,
	\end{split}
\end{equation}
where $\alpha$ is a positive constant. The right hand side of (\ref{{O4.2}}) can be estimated by the following:

\begin{equation}\label{O4.3}
	\begin{split}
		A_1+A_2\leq & \frac{\alpha }{4}\int_{-\infty}^{\infty}{\phi}^{2}{U}_{x}\mathrm{d} x + \frac{2{\mu }^{2}}{\alpha }\int_{-\infty}^{\infty}\frac{{U}_{xx}^{2}}{{U}_{x}}\mathrm{d} x + \frac{2{\gamma }^{2}}{\alpha }\int_{-\infty}^{\infty}\frac{{U}_{xxx}^{2}}{{U}_{x}}\mathrm{d} x.\\
	\end{split}
\end{equation}
Now, we estimate the error term $A_3$. By directly calculation, we have
\begin{equation*}
	\begin{split}
		&[f(\phi+U^{r})-f(\phi+U)]_x\\
		=&[f'(\phi+U^{r})-f'(\phi+U)] (\phi_x +U_x)+f'(\phi+U^{r})(U^{r}-U)_x\\
		=&[f''(\eta_1)(\phi_x+U_x)(U^{r}-U)]+f'(\phi+U^{r})(U^{r}-U)_x\\
	\end{split}
\end{equation*}
where $\eta_1$ between $\phi+U^{r}$ and $\phi+U$. From the priori assumption $\phi\left( {x,t}\right)  \in  {X}_{2\delta}\left( {0,T}\right)$, we have $\| \phi{\| }_{1} \leq  {2\delta}$. Further, $\left\| \phi\right\|_{L^{\infty}}  \leq  {2\delta }$  In addition, $f\left( u\right)$ is a smooth convex function on $\mathbb{R}$, so we can set
\begin{equation}\label{O4.4}
	F\left( \delta \right)  = \max \left\{  {\mathop{\sup }\limits_{{\left| \phi\right|  \leq  {2\delta }}}\left| {{f}^{\prime }\left( {\phi + U}\right) }\right| ,\mathop{\sup }\limits_{{\left| \phi\right|  \leq  {2\delta }}}{f}^{\prime \prime }\left( \theta \right) }\right\}  .
\end{equation}
It is easy to see that $F\left( \varepsilon \right)$ is a constant depending on $\varepsilon $. Thus one gets that
\begin{equation}\label{O4.5}
	\begin{split}
		&[f(\phi+U^{r})-f(\phi+U)]_x\\
		\leq& C (|U_x||U^{r}-U|+|\phi_x||U^{r}-U|+|(U^{r}-U)_x|),
	\end{split}
\end{equation}
where $\eta_1$ between $\phi+U^{r}$ and $\phi+U$. Similar, we have
\begin{equation*}\label{x210}
	[f({}U^{r})-f({}U)]_x \leq  C (|U_x||U^{r}-U|+|\phi_x||U^{r}-U|+|(U^{r}-U)_x|)
\end{equation*}
With the aid of lemma \ref{yl4}, the last term of  (\ref{{O4.2}}) can be estimated as
\begin{equation}\label{O4.6}
	\begin{split}
		A_3\leq & \| \phi\|_{L^{\infty}}     \| E\|_{L^{1}}\\
		\leq& C   \left( \| h\|_{L^{1}}+  \sum_{j=0} ^{3}{\left\|\frac{\partial^{j}}{\partial x^{j}}(U-U^{r})\right\|}_{{L}^{1}}+{\left\| (U-U^{r})\right\|\|U_x\|}+\left\| (U-U^{r})\right\|\|\phi_x\|    \right) \\
		\leq& C  {\epsilon }{e}^{-{\beta t}},
	\end{split}
\end{equation}
where we have used lemma \ref{yl2.1} and the Holder inequality in the last inequality. Combining (\ref{O4.3}) with (\ref{O4.6}), we can obtain from (\ref{{O4.1}})
\begin{equation}\label{O4.7}
	\begin{split}
		&\frac{1}{2}\frac{\mathrm{d}}{\mathrm{d}t}\int_{-\infty}^{\infty}{\phi}^{2}\mathrm{d} x + \int_{-\infty}^{\infty}{U}_{x}{\phi}^{2}\mathrm{d} x + \int_{-\infty}^{\infty}{\phi}_{x}^{2}\mathrm{d} x\\
		\leq&  C\left( {\int_{-\infty}^{\infty}\frac{{U}_{xx}^{2}}{{U}_{x}}\mathrm{d} x + \int_{-\infty}^{\infty}\frac{{U}_{xxx}^{2}}{{U}_{x}}\mathrm{d}x}+  {\epsilon }{e}^{-{\beta t}}\right) .\\
	\end{split}
\end{equation}
Integrate (\ref{O4.7}) with respect to $t$ over $\left\lbrack  {0,t}\right\rbrack$ and using Lemma  \ref{yl2.1}, we obtain the proof of this lemma.
\end{proof}

\section{Appendix}
\subsection {Proof of Lemma  \ref{yl4}}
\begin{proof}
	By directly calculation, using $\eqref{x2.6}$, we have

	\begin{equation*}
		\begin{split}
			& U-U^{r}=\frac{U^{r}-\bar{u}_{r}}{\bar{u}_{l}-\bar{u}_{r}}(u_l-\bar{u}_l)+\frac{\bar{u}_{l}-{U}^{r}}{\bar{u}_{l}-\bar{u}_{r}}(u_r-\bar{u}_r),
		\end{split}
	\end{equation*}
	and
	\begin{equation*} 
		\begin{split}
			h{=}&{U}_{t} + f\left( U\right)_x\\
			= &{\left[  {u}_{l}g + {u}_{r}\left( 1 - g\right) \right]  }_{t} + f\left( U\right)_x\\
			=& {u}_{l}{g}_{t} + {u}_{lt}g + {u}_{lt}\left( {1 - g}\right)  - {u}_{l}  {{g}_{t}}   + f\left( U\right)_x\\
			=& \left( {{u}_{l} - {u}_{r}}\right) {g}_{t} + \left( {{u}_{lt} - {u}_{rt}}\right) g + {u}_{lt} + f\left( U\right)_x\\
			=&\left( {{u}_{l} - {u}_{r}}\right) \frac{{U}^{r}_{t}}{\bar{u}_{l} - \bar{u}_{r}} + \left( {{u}_{l} - {u}_{r}}\right)_t g + \left[ {f\left( U\right) -f(u_l)  }\right]_x\\
			=& -\frac{\left( {u}_{l} - {u}_{r}\right) }{\bar{u}_{l} - \bar{u}_{r}}\left( {f\left( {U}^{r}\right) }\right)_x + \left( {{u}_{l} - {u}_{r}}\right)_t g + \left[ {f\left( U\right) -f(u_l)  }\right]_x\\
			=&  -\frac{\left( {u}_{l}-\bar{u}_{l}+\bar{u}_{r} - {u}_{r}\right) }{\bar{u}_{l} - \bar{u}_{r}}\left( {f\left( {U}^{r}\right) }\right)_x - \left( {f({u}_{l}) - f({u}_{r})}\right)_x g + \left[ {f\left( U\right) -f(u_l) -f(U^{r}) }\right]_x\\
			=&  -\frac{\left( {u}_{l}-\bar{u}_{l}+\bar{u}_{r} - {u}_{r}\right) }{\bar{u}_{l} - \bar{u}_{r}}\left( {f\left( {U}^{r}\right) }\right)_x \\
			&- \left[{f'({u}_{l})}\left(u_{l}-\bar{u}_{l} \right)_x-{f'({u}_{r})}\left(u_{r}-\bar{u}_{r} \right)_x     \right] g\\
			& + \left[ f'\left( U\right)(U-U^{r})_{x}\right]+     [f'(U) -f'(U^{r})] (U^{r})_x-     f'(u_l)(u_l-\bar{u}_l)_{x}. \\
		\end{split}
	\end{equation*}
	Thus, one gets that
	\begin{equation*}
		|U-U^{r}|,|h|\leq C      \left(    |(u_r-\bar{u}_r)|+ |(u_l-\bar{u}_l)|  +|(u_r-\bar{u}_r)_{x}|+ |(u_l-\bar{u}_l)_{x}|     \right).
	\end{equation*}
	By directly calculation, with the aid of Lemma \ref{yl2}, the lemma can be prove immediately.

\end{proof}

\end{document}